\documentclass[12pt]{amsart}
\usepackage{amsmath, amssymb, amsthm,  epsfig}
\theoremstyle{plain}
\newtheorem{thrm}{Theorem}[section]
\newtheorem*{thrm*}{Theorem}
\newtheorem{lemma}[thrm]{Lemma}
\newtheorem{prop}[thrm]{Proposition}
\newtheorem{cor}[thrm]{Corollary}
\theoremstyle{definition}
\newtheorem{dfn}[thrm]{Definition}
\theoremstyle{remark}
\newtheorem{rmrk}[thrm]{Remark}
\numberwithin{equation}{section}

\begin{document}

\newcommand{\R}{\mathbb R}
\newcommand{\N}{\mathbb N}
\newcommand{\C}{\mathbb C}
\newcommand{\lie}{\mathcal G}
\newcommand{\hN}{\mathcal N}
\newcommand{\D}{\mathcal D}
\newcommand{\A}{\mathcal A}
\newcommand{\B}{\mathcal B}
\newcommand{\sL}{\mathcal L}
\newcommand{\sLi}{\mathcal L_{\infty}}

\newcommand{\eps}{\epsilon}
\newcommand{\al}{\alpha}
\newcommand{\be}{\beta}
\newcommand{\p}{\partial}  

\def\dist{\mathop{\varrho}\nolimits}

\newcommand{\BCH}{\operatorname{BCH}\nolimits}
\newcommand{\Lip}{\operatorname{Lip}\nolimits}
\newcommand{\Hol}{C}                             
\newcommand{\lip}{\operatorname{lip}\nolimits}
\newcommand{\capQ}{\operatorname{Cap}\nolimits_Q}
\newcommand{\pCap}{\operatorname{Cap}\nolimits_p}
\newcommand{\Om}{\Omega}
\newcommand{\om}{\omega}
\newcommand{\half}{\frac{1}{2}}
\newcommand{\e}{\epsilon}
\newcommand{\vn}{\vec{n}}
\newcommand{\X}{\Xi}
\newcommand{\tLip}{\tilde  Lip}

\newcommand{\Span}{\operatorname{span}}

\newcommand{\ad}{\operatorname{ad}}
\newcommand{\Hn}{\mathbb H^n}
\newcommand{\Hone}{\mathbb H^1}
\newcommand{\Lie}{\mathfrak}
\newcommand{\Layer}{V}
\newcommand{\hgrad}{\nabla_{\!H}}
\newcommand{\im}{\textbf{i}}
\newcommand{\nz}{\nabla_0}

\newcommand{\ued}{u^{\e,\delta}}
\newcommand{\ueds}{u^{\e,\delta,\sigma}}
\newcommand{\tnabla}{\tilde{\nabla}}

\newcommand{\bx}{\bar x}
\newcommand{\by}{\bar y}
\newcommand{\bt}{\bar t}
\newcommand{\bs}{\bar s}
\newcommand{\bz}{\bar z}
\newcommand{\btau}{\bar \tau}

\newcommand{\LC}{\mbox{\boldmath $\nabla$}}
\newcommand{\Ne}{\mbox{\boldmath $n^\e$}}
\newcommand{\nuo}{\mbox{\boldmath $n^0$}}
\newcommand{\Xie}{X^\epsilon_i}
\newcommand{\Xje}{X^\epsilon_j}

\title{Generalized mean curvature flow in Carnot groups}

\author{Luca Capogna}\address{Department of Mathematical Sciences,
University of Arkansas, Fayetteville, AR 72701}\email{lcapogna@uark.edu}
\author{Giovanna Citti}\address{Dipartimento di Matematica, Piazza Porta S. Donato 5,
40126 Bologna, Italy}\email{citti@dm.unibo.it}
\keywords{mean curvature flow, sub-Riemannian geometry, Carnot groups\\
The authors are partially funded by NSF Career grant DMS-0124318 (LC) and GALA project (GC)}

\begin{abstract} In this paper we study   the generalized mean curvature flow of sets in
the sub-Riemannian geometry  of Carnot groups. We extend to our context the level sets method
and the weak (viscosity) solutions
introduced in the Euclidean setting in \cite{es:mc1} and \cite{CGG}. We establish two special cases of the
comparison principle, existence, uniqueness and basic geometric properties of the flow.
\end{abstract}
\maketitle

\section{Introduction}

The evolution of hypersurfaces with normal velocity given by
the mean curvature $K$ arises as the $L^2$ gradient flow
of the Riemannian perimeter functional.
A  detailed list of references concerning the study of the mean curvature
flow can be found in the monographs \cite{Ecker} and  \cite{zhu}.

 Although the mean curvature flow is locally smoothing, even starting with a smooth manifold as initial data, its flow  may develop singularities
before the extinction time, as in the famous example of the dumbbell
in \cite{es:mc1}.  Several methods have been suggested
in order to study the behavior of the flow past the
formation of singularities: the method
of currents  introduced by Brakke  \cite{Brakke},
 the method of generalized (viscosity) solutions  indipendently
developed by Chen, Giga and Goto \cite{CGG}, and by Evans and Spruck
 \cite{es:mc1}, \cite{es:mc2}, \cite{es:mc3}, \cite{es:mc4}, (see also the generalization
 by Ishii and Souganidis \cite{Ishi-Souganidis}),  De Giorgi's method of barriers \cite{degiorgi}
(which was studied in detail by Bellettini and Novaga \cite{bn:1997},  \cite{bellettininovaga2}) and the closely related definition by Barles and Souganidis
\cite{B-S} and \cite{Barles-Souganidis}.

Most pertinent to the present paper is the work in \cite{es:mc1} where,
 following \cite{oshersethian}
the authors
study the  flow  of  level sets $M_t=\{x\in \R^n\ | \ u(x,t)=0\}$
where the function $u$ is a generalized solution of the degenerate quasilinear, non-divergence form  PDE
\begin{equation}\label{eucPDE}
\p_t u (x,t)=K |\nabla u|=\sum_{i,j=1}^n \bigg(\delta_{ij}-\frac{\p_{x_i}u \p_{x_j}u}{|\nabla u|^2}\bigg)\p_{x_i \ x_j}^2 u.
\end{equation}
Note that the  PDE becomes degenerate exactly at the  singularity points of the level sets,
that is  where $|\nabla u|$ vanishes. The level set  approach was extended by   Ilmanen \cite{MR1189906})
to include the study of generalized flow of subsets in Riemannian manifolds.

The Riemannian mean curvature flow  has been used both as
a model for the study of sharp-interfaces in material science
and in digital image processing.  Recently the  first layer of the mamalian visual cortex
has been modelled as a smooth surface with a sub-Riemannian geometry (\cite{V:Hoffman} and \cite{V:PT}).
In this setting
some perceptual phenomena such as the formation of subjective surfaces,
are described as   sub-Riemannian mean curvature flows and
minimal surfaces (see  \cite{V:P}, \cite{CS},  \cite{SCM}, \cite{HP}, and \cite{HP3}).

The focus of the present paper is to study a model case
 of the  sub-Riemannian   analogue of the   mean curvature
flow: the  {\it horizontal mean curvature flow
in Carnot groups}.

Sub-Riemannian geometry is an extension of Riemannian geometry in which,
given a manifold $G$,
a metric $g_0$ is
only prescribed on a sub-bundle $HG\subset TG$ (called {\it horizontal} bundle).
The horizontal bundle is supposed to have
the so-called {\it bracket generating} properties, i.e. there is a number
$r\in \N$ such that all sections of $TG$ are generated by linear combinations
of sections
of $HG$ and their commutators up to order $r$. In a standard fashion one can associate
a control distance $d_C$ (the {\it Carnot-Carath\'eodory distance}) to the sub-Riemannian
structure  $(G, HG ,g_0)$.
Blow-up of such geometric structures (see \cite{gro:cc}, \cite{mitchell} and \cite{Roth:Stein}) give rise to ``linear" sub-Riemannian
manifolds called Carnot groups (nilpotent Lie groups endowed with
a control metric, see the next section for a precise definition).
In a sense, Carnot groups are the model for
the tangent spaces to sub-Riemannian manifolds.

Sub-Riemannan structures can also be seen as degenerate limit
of Riemannian structures: Fix an orthonormal frame $\mathcal F_0=\{X_1,...,X_m\}$ of sections
of $HG$ and extend it to a frame $\mathcal F$  of $TG$.
 Define Riemannian metrics $g_{\e}$, $\e>0$,
extensions of $g_0$ to all of $TG$, such that at each point, the length of any
non-horizontal section
in $\mathcal F$ is $\e^{-1}$. If we denote by $d_\e$ the distance function
associated to $g_\e$ then $(G ,d_\e)\to (G, d_C)$ in the sense of
the Gromov-Hausdorff convergence between metric spaces (see \cite{gro:metric},\cite{gro:cc}, and
\cite{Montgomery:book}).
For a more in-depth presentation  of  Sub-Riemannian geometry
 we refer the reader to   \cite{ste:harmonic}, \cite{gro:metric},\cite{gro:cc},
\cite{Montgomery:book}, \cite{cdpt:survey} and  references therein.

If $G$ is a Carnot group and $M\subset G$ is a smooth hypersurface
we define $\Sigma(M)$ the set of characteristic points
of $M$, i.e. the points $x\in M$ where the horizontal structure is
contained in the tangent space. Derridij \cite{derridj} proved that $\Sigma(M)$ has
zero surface measure, this result was later refined in \cite{bal:characteristic} and \cite{magnani2}.
 Outside $\Sigma(M)$ one can define
a horizontal normal $\nuo$, the normalized
 projection onto $HM$ of the Riemannian normal
(in any of the metrics $g_\e$). Accordingly, the  {\it horizontal }  mean
curvature $K_0$  can be defined
as the first variation of the sub-Riemannian perimeter
in the horizontal normal direction (see \cite{dgn:minimal},
 \cite{BC}, \cite{RR}, \cite{pau:cmc-carnot} and \cite{selby})\footnote{
Closely linked to the study of mean curvature flow, the analysis of
minimal surfaces in the sub-Riemannian setting has recently seen great activity
\cite{gn:isoperimetric}, \cite{pau:minimal}, \cite{chmy:minimal}, \cite{chy}, \cite{gp:bernstein}, \cite{dgn2}, \cite{bascv} and \cite{ni:cmc}).}.
Generically such curvature is unbounded in a neighborhood of $\Sigma(M)$
and cannot be defined at characteristic points.

The {\it horizontal mean curvature flow} of a hypersurface
of a Carnot group $G$
is the flow $t \rightarrow M_t \subset G$
in which each point $x(t)\notin \Sigma(M_t)$ in the evolving manifold
moves along the horizontal normal
with speed given by the horizontal mean curvature. The corresponding equation,
outside the characteristic set, is
\begin{equation}\label{pde-flow}  \frac{d x}{dt} = - K_0 \nuo.\end{equation}
Extending the techniques  in \cite{es:mc1},
the evolving surface $M_t$ can be represented as zero level set
of a function $u(x,t)$ which solves
the PDE,

\begin{equation}\label{PDE}
\p_t u (x,t)=\sum_{i,j=1}^m \bigg(\delta_{ij}-\frac{X_i u X_j u}{\sum_{i=1}^m (X_i u)^2}\bigg)X_i X_j u.
\end{equation}

There is an obvious immediate difficulty in the study of this equation: It is not
well defined in $\Sigma(M_t)$. In contrast with the Euclidean setting the
PDE becomes degenerate not only at singularities of the level sets, where the full (spatial) gradient  of the solution $\nabla u(\cdot, t)$  vanishes, but also at characteristic points.
It is in fact only the vanishing  of the horizontal portion of the full gradient which determines
the characteristic set. In a sense, points in $\Sigma(M_t)$ correspond to {\it metric} singularities in the set $M_t$.

In view of this new difficulty, in the present paper while we are
able to prove existence for the general flow described in
\eqref{PDE}, at the moment we can prove comparison principles and
uniqueness only for a special class of flows, i.e. either in the
presence of particular classes of initial data or for  {\it graphs }
over Carnot groups.

%
%
%
%

Bonk and the first named author study in \cite{BC} some  properties
of smooth solutions of this equation. In that paper the solution is interpreted
in the vanishing viscosity sense, i.e.
limit of Riemannian mean curvature flows. However they
assume the existence of such
smooth vanishing viscosity solution.

For elliptic or parabolic PDE the notion of vanishing viscosity is
equivalent to the notion of viscosity solution
(see \cite[Section 6]{usersguide}). This question has
not yet been addressed in the sub-Riemannian setting, where
however a number of authors have studied viscosity solutions
for non degenerate PDE: \cite{bieske}, \cite{bieske2},
\cite{wang:aronsson}, \cite{wang:convex}, \cite{lms},
\cite{baloghrickly}, and  \cite{magnani:convex}.

In the present paper we give a new definition of continuous (non smooth) viscosity solutions
to \eqref{PDE}. The novelty of our definition comes from the fact that the equation is totally degenerate at  characteristic points, where the right hand side of
\eqref{PDE} is not defined.
While we cannot  prove that such viscosity solutions are equivalent to vanishing viscosity solutions
we  establish existence and uniqueness
of Lipschitz vanishing viscosity solutions for  the initial value problem, and some basic geometric properties of the flow.
Hence our results can be interpreted as special limit cases of Ilmanen's work \cite{MR1189906} in the approximation $g_\e\to g_0$ described above.

\medskip

The paper is organized as follows:
In section 2 we give the definition of viscosity solutions  to \eqref{PDE} and recall some
existing results.
In section 3 we prove two particular cases of the  comparison theorem between a bounded subsolution $u$
and a supersolution $v$ of equation \eqref{PDE}, from which uniqueness follows.
In order to do so it is quite standard to use as test function
the difference of regularized versions of these
two functions (the so called sup inf- convolutions) in two different points,
with a penalization term.
The choice of this test function and in particular of the
penalization term depends crucially on the sub-Riemannian character of the
problem.
In Section 4 we identify simple classes of solutions
(self-shrinking cylinders and stationary planes)
and construct bounded barriers which will be used in the
proof of existence and in the study of geometric properties of
the flows. The construction of explicit solutions is not trivial in our setting.
The metric sphere, which  is self-shrinking
and heavily used in the Euclidean setting (see \cite{es:mc1}),
does not have a self-similar evolution in our setting.
Indeed it is an open question wether  there is any closed manifold which
gives rise to a self-similar solution. See \cite{BC} for a study of
self-similar solutions in the Heisenberg group.

In Section 5 we prove the
existence of Lipschitz vanishing viscosity solutions, and the fact that they are also viscosity solutions. In the proof of existence we first provide higher order a priori estimates
for the solutions of the approximating Riemannian flows.
We cannot rely on the estimates proved by Ilmanen in \cite{MR1189906}
as they are dependent on curvature bounds, which fail in our setting.
Moreover the non-commutativity of the vector fields $X_i$ makes it hard to prove
a-priori higher order estimates. We deal with this problem by using both
left-invariant and right-invariant derivatives.
Indeed, such left and right derivatives commute (by definition), allowing to
easily differentiate the equation. This, along with a parabolic
maximum principle yields the desired bounds.

In Section 6 we prove some simple geometric properties of the evolution.
Lacking a complete comparison principle  we cannot show that the generalized flow does not depend on the choice of the initial defining
function, but only  its zero level set. We show that  if two sets $M, \hat M$ satisfy $M\subset \hat M$, then the inclusion $M_t\subset \hat M_t$ between their evolutions $M_t$, $\hat M_t$ persists for all
times. Since this result depends on the comparison principle we need some additional hypothesis on $\hat M_t$.

For generalized level sets arising out of vanishing viscosity solutions, we show also that the right invariant control distance between two disjoint initial
sets increases in the evolution. As a corollary we have that any initial compact
set  has a finite extinction time, i.e. the evolving set shrinks and
eventually vanishes in a finite time.

 To conclude, we recently learned that Dirr, Dragoni and Von Renesse
 have recently studied a probabilistic approach to the mean curvature
 flow in the context of the Heisenberg group in the same spirit of \cite{SonerTouzi}.

{\it Acknowledgments.} Part of the work on this paper was done while the authors were
guests of the Centro di Ricerca Matematica Ennio de Giorgi, in Pisa, Italy.
We thank the staff of the center, M. Giaquinta, F. Ricci and
L. Ambrosio for their hospitality and for interesting conversations.

\section{Definitions and preliminary results}
\subsection{Carnot group structure}
Let $G$ be an analytic and simply connected Lie group with
topological dimension $n$ and such that its Lie algebra $\lie$
admits a stratification $\lie=V^1\oplus V^2 \oplus ...\oplus V^r$,
where $[V^1,V^j]=V^{j+1}$, if $j=1,...,r-1$, and $[V^k,V^r]=0$,
$k=1,...,r$. Such groups are called in \cite{fol:1975},
\cite{fs:hardy}, and \cite{ste:harmonic} {\sl stratified nilpotent
Lie groups}. Fix $X_{1},...,X_{m}$ a basis of $V^1$,
 called the horizontal frame, and complete it to a basis $(X_1,...,X_{n})$ of $\lie$
by choosing for every $k=2, \cdots r$ a basis of $V_k$.  If $X_i$
belongs to $V_k$, then we will set $d(i)=k$. We will denote by $xX
= \sum_{i=1}^{n} x_{i} X_{i}$ a generic element of $\lie$.
 Since the exponential map $\exp:\lie \to G$ is a
global diffeomorphism we use exponential coordinates in $G$, and
denote $x=(x_1, \cdots, x_n)$ the point $\exp \big(xX\big).$ We
also set $x_H = (x_1, \cdots, x_m)$ and
$x_V=(x_{m+1},...,X_n)$ so that $x=(x_H,x_V)$.   Define non-isotropic
dilations as $\delta_s(x)= (s^{d(i)} x_{i})$, for $s>0$.

We denote by $(X_1,...,X_{n})$ (resp. $(\tilde X_1,...,\tilde
X_{n})$) the left invariant (resp. right invariant) translation of
the frame $(X_1,...,X_{n})$ of $\lie$.
 Set $H(0)= V^1$, and for any $x\in G$ we let
$H(x)=xH(0)=\text{span}[X_{1},...,X_{m}](x)$. The distribution
$x\to H(x)$ is called {\it the horizontal sub-bundle} $H$. On $H$
we define a left invariant positive definite form $g_0$, so that
$X_1\cdots, X_m$ is an orthonormal frame.  We let
$\nabla=(X_1,\cdots,X_m)$ denote the {\it horizontal gradient}
operator. The vectors $X_{1}....X_{m}$ and their commutators span
all the Lie algebra
 $\lie$, and consequently verify H\"ormander's finite rank condition
(\cite{hormander}). This allows to use the results from
\cite{nsw}, \ and define a control distance $d_C(x,y)$
 associated to the distribution $X_{1}....X_{m}$, which is
called {\sl the
 Carnot-Carath\'eodory metric} (denote by $\tilde d_C$
 the corresponding right invariant distance). We
 call the couple $(G,d_C)$ a {\it Carnot Group}.

We define a family of left invariant Riemannian metrics $g_{\e}$,
$\e>0$ in $\lie$ by requesting that  $\{X_1,\cdots, X_m. \e
X_{m+1}
 ,\cdots, \e X_n\}$ is an orthonormal
frame. We will denote by $d_{\e}$ the corresponding distance
functions. Correspondingly we use $\nabla_\e,$ (resp. $\tilde
\nabla_\e$) to denote the left (resp. right) invariant gradients.

  It is well known\footnote{See for instance \cite{gro:metric}} that $(G,d_{\e})$  converges
in the Gromov-Hausdorff sense as $\e\to 0$ to the sub-Riemannian
space $(G,d_C)$. The Carnot-Carath\'eodory metric  is equivalent
to a more explicitly defined pseudo-distance function, that we
will call (improperly) {\sl gauge distance}, defined as
$$|x|^{2r!}=\sum_{k=1}^r \sum_{i=1}^{m_k} |x_{i,k}|^{\frac{2r!}{k}}, \text{ and
} d(x,y)=|y^{-1}x|$$ If $x\in G$ and $r>0$, we will denote by
 $B(x,r)=\{y\in G \ |\ d(x,y)<r\}$ the   balls
in the gauge distance.

We recall now the expression of the left invariant vector fields
in exponential coordinates  (see \cite{Roth:Stein})
\begin{equation}\label{bch}
X_i = \p_i + \sum_{k=d(i)+1}^r\sum_{d(j) = k}p_{ik}^j(x)\p_j,
\end{equation}
where $p_{ik}^j(x)$ is an homogeneous polynomial of degree
$k-d(i)$ and depends only on $x_h$, with $d(1)\leq d(h) \leq
k-d(i)$.


\subsection{Horizontal mean curvature  flow of hypersurfaces}

Let $M\subset G$ be a $C^2$ smooth hypersurface, denote by $\Ne$
the unit normal in the metric $g_\e$ and by
$\nuo=\sum_{d(i)=1}(\nuo)_i X_i$ its normalized projection in the
$g_\e$ norm onto the horizontal plane. Note that this is not
dependent on $\e$ and is well defined only outside the
characteristic set $\Sigma(M)=\{x\in M|\ H(x)\subset T_xM\}$.
 The vector
$\nuo$ is called {\it horizontal normal} and its (horizontal)
divergence
\begin{equation}\label{hmc}
K_0= \sum_{d(i)=1} X_i \nuo_i
\end{equation}
is known as the horizontal mean curvature of $M$ at $x\notin \Sigma(M)$. Note that
even for smooth (in the Euclidean sense) hypersurfaces
the horizontal mean curvature may blow up near characteristic points.

We study the flow $t\to M_t$ where a point $x\in M_t$ evolves with
velocity $\p_t x=-K_0 \nuo$. The level set approach consists in
studying a PDE describing the evolution of a function $u(x,t)$
such that\footnote{When a manifold is defined as a level set, we
tacitly assume that the gradient of the defining function does not
vanish in a neighborhood of the manifold.} $M_t=\{x\in G|\
u(x,t)=0\}$. In this setting one has $\Ne=\nabla_\e u/|\nabla_\e
u|$ and $\nuo=\nabla_0 u/|\nabla_0 u|$. Consequently, on a formal
level,
 one has
\begin{multline}\label{flow}
\p_t u(x(t),t)=<\nabla_{0} u(x(t)), \p_t x(t)>_0+\p_t u(x,t)\\=
-K_0 <\nabla_0 u,\nuo>+\p_t u=-K_0|\nabla_0 u|+\p_t u=0.
\end{multline}

This problem is well approximated by the Riemannian mean curvature
flows $\p_t x=- K_{\e}\Ne$, where $K_\e=\sum_{i=1}^n \Xie \Ne_i$
is the $g_\e$ mean curvature of $M$. The corresponding evolution
PDE for the level sets is $\p_t u^{\e}= K_\e |\nabla_\e u|$. We
observe that for a given hypersuface, $\Ne\to \nuo$ and $K_e\to K_0$ as $\e\to 0$, outside the characteristic set.
We will prove in Section 5 that $u^{\e}\to u$ weak solution of
\eqref{flow}.

A simple computation shows that the mean curvature $K_\e$ of the
manifold $\{u(x)=0\}$ is given by the identity
$$K_\e|\nabla_\e u| = \sum_{i,j=1}^n \bigg(\delta_{ij}-\frac{\Xie u \Xje u}{|\nabla_\e u|^2}\bigg)
\Xie \Xje u, $$ Outside of the characteristic set the horizontal
mean curvature $K_0$ is expressed as
$$K_0|\nabla_0 u| = \sum_{i,j=1}^m \bigg(\delta_{ij}-\frac{X_i u X_j u}{|\nabla_0 u|^2}\bigg)
X_iX_j u.$$

Consequently \eqref{flow} can be rewritten more explicitly as
\begin{equation}\label{MAIN}
u_t=\sum_{i,j=1}^m \bigg(\delta_{ij}-\frac{X_i u X_j u}{|\nabla_0
u|^2}\bigg) X_iX_j u, \text{ for }x\in G, t>0.
\end{equation}
 If the Carnot group is a product $G = \tilde G\times \R$ and we use coordinates
 $(x,e)\in \tilde G\times \R$, then
from \eqref{MAIN} and by representing the function $u$ as
$u(x,e,t)=e-U(x,t)$, we obtain
 a special class of evolutions, given by
 graphs over $\tilde G$ of the form $M_t=\{(x, U(x,t))\ |\ x\in \tilde G, \ t>0\}$
 where $U: \tilde G\to \R$ is a solution of
 \begin{equation}\label{MAINbis}
U_t=\sum_{i,j=1}^m \bigg(\delta_{ij}-\frac{X_i U X_j U}{1+|\nabla_0
U|^2}\bigg) X_iX_j U, \text{ for }x\in \tilde G, t>0.
\end{equation}
Note that such graphs are always non-characteristic.

\subsubsection{Weak solutions.} As in the Euclidean case, one cannot expect the smoothness of the
solution to be preserved for all times. Moreover, even for smooth
solutions, the horizontal gradient vanishes at all characteristic
points making the equation degenerate. To overcome these
difficulties we use the subelliptic analogue of viscosity
solutions (see also for earlier related definitions
\cite{bieske2},\cite{wang:convex}).

\begin{dfn}
A function $u\in C(G\times [0,\infty)$ is a weak subsolution of
\eqref{MAIN} in $G\times (0,\infty)$ if for any $(x,t)\in G\times
(0,\infty)$ and any function $\phi\in C^2(G)\times (0,\infty)$
such that $u-\phi$ has a local maximum at $(x,t)$ then

\begin{equation}\label{compabove}
\p_t \phi \le  \begin{cases}
\sum_{i,j=1}^m\bigg(\delta_{ij}-\frac{X_i \phi X_j \phi}{|\nabla_0
\phi|^2}\bigg)
X_iX_j \phi & \text{ if }|\nabla_0 \phi| \neq 0\\
 \sum_{i,j=1}^m(\delta_{ij}-p_i p_j)X_iX_j \phi & \text{ for some }p\in \R^m,
|p|\le 1, \text{ if }|\nabla_0 \phi| = 0 .\end{cases}
\end{equation}

A function $u\in C(G\times [0,\infty)$ is a weak supersolution of
\eqref{MAIN} if
\begin{equation}\label{compbelow}
\p_t \phi \geq  \begin{cases}
\sum_{i,j=1}^m\bigg(\delta_{ij}-\frac{X_i \phi X_j \phi}{|\nabla_0
\phi|^2}\bigg)
X_iX_j \phi & \text{ if }|\nabla_0 \phi| \neq 0\\
 \sum_{i,j=1}^m(\delta_{ij}-p_i p_j)X_iX_j \phi & \text{ for some }p\in \R^m,
|p|\le 1, \text{ if }|\nabla_0 \phi| = 0.\end{cases}
\end{equation}

A weak solution of \eqref{MAIN} is a function $u$ which is both a
weak subsolution and a weak supersolution.
\end{dfn}

In the graph case $G = \tilde G \times \R$ when we consider only
evolving surfaces of the form $M_t = \{e = U(x), x \in \tilde G\},$
we can also reduce the class of test functions in the previous
definition to those of the form $\phi(e,x) = e - \psi(x)$. In this
way the definition of viscosity solutions becomes

\begin{dfn}
A function $U\in C(\tilde G\times [0,\infty)$ is a weak subsolution
of \eqref{MAINbis} in $\tilde G\times (0,\infty)$ if for any
$(x,t)\in \tilde G\times (0,\infty)$ and any function $\psi\in
C^2(\tilde G)\times (0,\infty)$ such that $U-\psi$ has a local
maximum at $(x,t)$ then

\begin{equation}\label{compabovebis}
\p_t \psi \le \sum_{i,j=1}^m\bigg(\delta_{ij}-\frac{X_i \psi X_j
\psi}{1+|\nabla_0 \psi|^2}\bigg) X_iX_j \psi
\end{equation}

A function $U\in C(\tilde G\times [0,\infty)$ is a weak
supersolution of \eqref{MAINbis} if
\begin{equation}\label{compbelowbis}
\p_t \psi \geq \sum_{i,j=1}^m\bigg(\delta_{ij}-\frac{X_i \psi X_j
\psi}{1+|\nabla_0 \psi|^2}\bigg) X_iX_j \psi
\end{equation}

A weak solution of \eqref{MAINbis} is a function $U$ which is both a
weak subsolution and a weak supersolution.
\end{dfn}

As in \cite{crandall}, \cite{Ishii}, in the Euclidean setting and
\cite{bieske2} in the Heisenberg group, we have an equivalent
definition of weak sub(super)solutions.

\begin{dfn}\label{jet}
A function $u\in C(G\times [0,\infty))\cap L^{\infty}(G\times
[0,\infty))$ is a weak sub-solution of equation \eqref{MAIN} if
whenever $(x,t)\in G\times [0,\infty)$ for every $yX\in \lie$ and
$s\in \R$ and
\begin{eqnarray}\label{taylor-due}
u(\exp\big(yX\big)(x),t+s)  \le && u(x,t)+ \sum_{d(i) =1}^2 p_i y_i  \notag \\
&& \quad \quad +\frac{1}{2}\sum_{i,j=1}^m r_{ij} y_i y_j  + q s
+o( |y|^2 + s^2).
\end{eqnarray}
for some $p\in V^1\oplus V_2$, $q\in \R$ and $R=(r_{ij})\in
\R^{m\times m} $ then

\begin{equation}\label{punchline}
q\le\begin{cases}
 \sum_{i,j=1}^m \bigg(\delta_{ij}-\frac{p_i p_j}{|p_H|^2}\bigg) r_{ij} \quad \text{ if } |p_H|\neq 0\\
 \sum_{i,j=1}^m (\delta_{ij}-\eta_i\eta_j)r_{ij} \quad \text{ for some }|\eta|\le
1 \text{ if } |p_H|=0.\end{cases}
\end{equation}
\end{dfn}

\subsubsection{Generalized flow.} The evolution of an initial bounded hypersuface $M_0\subset G$ is
described in the following way: Choose a bounded function $f\in
C(G)$ such that $M_0=\{f(x)=0\}$.  We define the generalized
horizontal mean curvature flow $M_t$ of $M_0$ as the level sets
$M_t=\{u(x,t)=0\}$ for $u$ a weak solution of \eqref{MAIN}
satisfying the initial condition

\begin{equation}\label{initdata}
u(x,0)=f(x), \text{ for }x\in G.
\end{equation}

We  remark explicitly that this notion of generalized flow allows
for the evolution of any compact set, not necessarily an
hypersurface. In order for definition to make sense one needs to
show that the evolution does not depend on the choice of the
defining function $f$.  Lacking a suitable form of comparison principle
we will not be able to prove this. however,  we will establish
existence and special cases of the comparison principle, leading
to the basic geometric property of finite time extinction.

\subsection{Preliminary results.} In order to study weak solutions of \eqref{MAIN} we need the
subelliptic analogue of the so called sup-inf convolution as defined
in \cite{wang:convex}.

\begin{dfn}\label{defsupinf}
For $\epsilon>$ and $u: \R^n \to \R$ an upper semicontinous and
bounded from below function, the sup-convolution $u^{\mu}$ of $u$
is defined by
$$u^{\mu} = \sup_{G}\big(u(y) - \frac{1}{2\mu}|y^{-1}x|^{2r!}
\big)\forall x\in G
$$
The inf-convolution $v_{\mu}$ of $u$ is defined as
$$u_{\mu} = \inf_{G}\big(u(y) + \frac{1}{2\mu}|y^{-1}x|^{2r!}
\big)\forall x\in G
$$
If $x\in G$ set $|x|^2_E = x_1^2 + \cdots + x_n^2.$ We will say
that $u$ is semiconvex if for some constant $C>0$ the function
$u(x) + C|x|^2_E $  is convex in the Euclidean sense.
\end{dfn}

We use this definition of semiconvexity as in one of our proofs we
will need to invoke Jensen maximum principle in the Euclidean
setting.

\begin{lemma}\label{jensen}
If $f\in C(R^N)$ is semi-convex and achieves a local maximum
at the origin, then there exists a sequence $\{x^k\}_{k\in \N}$ converging to the origin,
such that:

(i) for each $k\in N$ the function $f$ is twice differentiable in the Euclidean sense at $x^k$

(ii) $|D_E f(x^k)| =o(1)$ as $k\rightarrow \infty$

(iii) $D^2_E f(x^k) \leq o(1) I_N$  as $k\rightarrow \infty$

where we have denoted by $D_E$ and $D^2_E$ respectively the Euclidean gradient and the Euclidean
Hessian, while $I_N$ is the identity $N\times N$ matrix.
\end{lemma}

This lemma is a refinement by Jensen \cite{jensen} of a result of Aleksandrov's.
The form in which we state it is from \cite[Lemma A.4]{usersguide}.

The following lemma due to Wang plays a crucial role in our
proofs:
\begin{lemma}\label{supinf}
An upper-semicontinous function $u:G\to \R$ satisfies
\begin{itemize}
  \item [i)] $u^{\mu}$ is semiconvex and locally Lipschitz
  continuous with respect to $d$.
  \item [ii)] $u^\mu$ is pointwise monotonically non decreasing in $\mu$ and
  converges to $u$.
  \item [iii)] if $u$ is a weak subsolution of \eqref{MAIN}, then
  so is $u^{\mu}$
\item [iv)] if $u$ is continuous then $u^{\mu}$ converges to $u$
uniformly on compact sets.
\end{itemize}\end{lemma}
Analogous  results hold for the inf-convolution $u_{\mu}$. For the
proof see \cite[Proposition 2.3]{wang:convex}.

\section{Comparison principles}

The analysis of the generalized mean curvature flow rests on
a comparison principle which roughly speaking should read as follows:
{\it If $u$ and $v$ are respectively a bounded, subsolution and supersolution of \eqref{MAIN}, and if $u(x,0)\le v(x,0)$ for all $x\in G$
and either $u$ or $v$ are uniformly continuous at time $t=0$, then
$u(x,t)\le v(x,t)$ for all $x\in G$ and $t\ge 0$.}

The sub-Riemannian geometry underlying our problem, in particular
the existence of characteristic points, makes such a comparison
principle much more difficult than its Euclidean counterpart
(see for instance
 \cite[Theorem 3.2]{es:mc1}). In this section we prove
 two special instances of such a comparison principle,
 namely in Theorem \ref{ESth32} we will consider functions $u$ and $v$ satisfying more restrictive assumptions at time $t=0$ and
 in Theorem \ref{ESth32bis} we will consider only graph-like
 solutions in a product group $G\times \R$.

  The main difference between the proof of our Theorem \ref{ESth32} and the corresponding Euclidean result is that
 the degeneration of the
PDE in the Euclidean setting occurs at points where the gradient
of the solution
vanishes. In the subriemannian setting for the degeneration to
occur it suffices that the horizontal componend of the gradient vanish. To deal
with this more singular phenomena we need a fine analysis of the
interplay between the stratification of the Lie algebra and the
properties of super and subsolutions.
%
%
%
%
%
%
%
%

\begin{thrm}\label{ESth32}
Assume that $u$ is a bounded weak subsolution and $v$ is a bounded
weak supersolution of \eqref{MAIN}. Suppose further

(i) For all $(x_H,x_V), (x_H,y_V)\in G$
$u(x_H,x_V,0)\leq v(x_H,y_V,0).$

(ii) Either $u$ or $v$ is uniformly continuous when restricted to

$G \times \{t=0\}$. Then $u(x,t)\leq v(x,t)$ for all $x\in G$ and
$t\ge 0$.
\end{thrm}

\begin{rmrk}
By choosing an appropriate barrier function we use the comparison
principle above it to prove the finite time extinction for compact
initial data.
\end{rmrk}

\begin{proof}
\medskip {\bf 1.} Should the thesis fail, then
for $\alpha >0$ small enough,
\begin{equation}\label{ES3.10} \max_{x,t}(u(x,t)-v(x,t)-\alpha t)\geq a/2 > 0.
\end{equation}
 Consequently, if we choose
$\mu > 0$ and sufficiently small,
\begin{equation}\label{ES3.12}
\max_{x,t}(u^\mu(x,t)-v_\mu(x,t)-\alpha t) \geq a/4 > 0.
\end{equation}
where the functions $u^\mu$ and $v_\mu$ denote respectively the
sup and inf convolutions of $u$ and $v$, defined as in
\eqref{supinf}.

\medskip {\bf 2.}
Given $\delta,\lambda>0$ define for $x\in G$, $yX \in \lie$ and $t,
t+s\in [0, +\infty[$
\begin{multline}\label{ES3.13}
\Phi(x,y,t,s) \\ \equiv u^\mu(y, s)- v_\mu(x, t) - \alpha t -
\delta^{-1}(|(x^{-1} y)_H|^4+ |s-t|^4) -
\lambda(|x|^{2r!}+|y|^{2r!}+|t|^2+|s|^2).
\end{multline}
We explicitly note that $(x^{-1} y)_H,$ simply reduces to the
standard Euclidean difference in the first layer  $V_1=\R^m$. In
view \eqref{ES3.12} we see
\begin{equation}\label{ES3.14}
\max_{x,y,t,s} \Phi(x,y,t,s)  \geq a/4 > 0.
\end{equation}
Choose a point  $(\bx,\by,\bt,\bs),$ so that

\begin{equation}\label{ES3.15}
\Phi(\bx,\by,\bt,\bs) = \max_{x,y,t,s} \Phi(x,y,t,s).
\end{equation}
Then \eqref{ES3.13} and \eqref{ES3.14} together with the
boundedness of $u^\mu$ and $v_\mu$, implies
\begin{equation}\label{ES3.16}  \lambda(|\bx|^{2r!}+|\by|^{2r!}+|\bt|^2+|\bs|^2) \leq C,
\ \ |(\bx^{-1} \by)_H|, \ |\bs-\bt|\leq C \delta^{1/4}.
\end{equation}
where $C>0$ is a constant independent of $\lambda$ and $\delta$.
We remark that \eqref{ES3.16} and the homogeneity of  the gauge
function implies that
\begin{equation}\label{olambda}
\nabla_0(\lambda |x|^{2r!})|_{x=\bx}= O(\lambda^{1/2r!}) \text{ and }
\nabla_0^2(\lambda |x|^{2r!})|_{x=\bx}= O(\lambda^{1/r!}),
\end{equation}
for $\lambda$ sufficiently small.

\medskip

{\bf 3.}
Arguing as in \cite{es:mc1} and using Wang's Lemma \ref{supinf},
we deduce now that
\begin{equation}\label{ES3.17}
\bt, \bs >\sigma(\mu)=c\sqrt{\mu}
\end{equation}
and that
\begin{equation}\label{ES3.18}
u^\mu \text{ is a weak subsolution}
\end{equation}
and
\begin{equation}\label{ES3.18b}
v_{\mu}\text{ is a weak supersolution}.
\end{equation}
Suppose that \eqref{ES3.17} does not hold, then $\bt, \bs \le
c\sqrt{\mu}$. Assuming $u(\cdot,0)$ is uniformly continuous we have
\begin{eqnarray}
0&&<a/4 \le \Phi(\bx,\by,\bt,\bs) \notag \\ &&\le u^\mu(\by,\bs)-v_\mu(\bx,\bt)
\le u(\by,\bs)-v(\bx,\bt) +o(1) \text{ as }\mu\to 0\ \text{ (in view of Lemma
\ref{supinf}) }\notag
\\ &&\le u(\by,0)-v(\bx,0) +o(1) \text{ as }\mu\to 0\ \text{ (in view of continuity )}\notag \\
&&=u((\by_H,\by_V),0) - v((\bx_H,\bx_V),0) +  o(1) \text{ as }\mu\to
0
\notag\\
&&\le u((\bx_H,\by_V),0)- v((\bx_H,\bx_V),0) +  o(1) \text{ as
}\mu\to 0 \notag \\ && \ \ \ \ \ \ \ \ \ \ \ \ \ \ \ \ \ \ \ \ \ \ \
\ \ \ \ \ \ \  \text{ and }
\delta \to 0 \ \ \ \text{ (in view of the uniform continuity of } u) \notag \\
&&\le o(1) \text{ (in view of assumption (i))}. \notag
\end{eqnarray}

\medskip {\bf 4.} Next, we show that $|\by^{-1}\bx|_H$
is bounded away from zero uniformly in $\lambda$. Using the fact
that $(\bx, \by, \bt, \bs)$ is a maximum point
\begin{multline}\label{previous}
 u^\mu(y, s)- v_\mu(x, t) - \alpha t -
\delta^{-1}(|(x^{-1} y)_H|^4+ |s-t|^4)-
\lambda(|x|^{2r!}+|y|^{2r!}+|t|^2+|s|^2) \\ \leq  u^\mu(\by, \bs)-
v_\mu(\bx, \bt) - \alpha \bt - \delta^{-1}(|(\bx^{-1} \by)_H|^4+
|\bs-\bt|^4)- \lambda(|\bx|^{2r!}+|\by|^{2r!}+|\bt|^2+|\bs|^2) .
 \end{multline}
Substituting  $x= \bx,$ $t= \bt,$ in the previous expression
yields
\begin{multline}\label{uno}
u^\mu(y, s)   \leq u^\mu(\by, \bs) \\ + \delta^{-1}\big(|(\bx^{-1}
y)_H|^4 - |(\bx^{-1} \by)_H|^4+ |s-\bt|^4- |\bs-\bt|^4\big)\\  +
\lambda (|y|^{2r!}-|\by|^{2r!}+|s|^2-|\bs|^2).
 \end{multline}
Choosing $z$ such that $y = \by z = \exp(z X)(\by)$  we see that
$(\bx^{-1} y)_H= (\bx^{-1}\by z)_H$.
 Observe that
 \begin{multline}\label{due}
 |(\bx^{-1}
y)_H|^4 - |(\bx^{-1} \by)_H|^4 = \sum_{i=1}^mf_iz_i +
\sum_{i, j=1}^m f_{ij}z_iz_j + o(|z|^2)
  \end{multline}
  with
 \begin{multline}
f_i=-4|(\by^{-1} \bx)_H|^2(\by^{-1} \bx)_i, \\ \text{ and }
f_{ij}=4\Big(|(\by^{-1} \bx)_H|^2 \delta_{ij} + 2 (\by^{-1}
\bx)_i(\by^{-1} \bx)_j\big)\ i,j=1,\cdots, m.
 \end{multline}
Moreover
 \begin{multline}\label{tre}
\lambda |y|^{2r!}-\lambda|\by|^{2r!}= \sum_{d(i)=1}^2k_iz_i +
\sum_{i, j=1}^m k_{ij}z_iz_j + o(|z|^2)
  \end{multline}
  with
\begin{multline}\label{ultima}
|k_i|\leq C\lambda |\by|^{2r!-d(i)}=O(\lambda^{1/r!}), \ \ \text{
and }
\\
|k_{ij}|\leq C\lambda |\by|^{2r!-2}=O(\lambda^{1/r!}),
\end{multline}
here we have used \eqref{olambda}. Substituting \eqref{due} -
\eqref{ultima} in \eqref{uno} we obtain
 \begin{multline}
u^\mu(\exp(z X)(\by), s) \\ \leq  u^\mu(\by, \bs)+
\delta^{-1}\big(\sum_{i=1}^mf_iz_i + \sum_{ij=1}^m f_{ij}z_iz_j
\big) + 4 \delta^{-1}(\bs-\bt)^3(s-\bs)\\ + 2\lambda \bs (s-\bs)+
\sum_{d(i)=1}^2k_iz_i + \sum_{ij=1}^m k_{ij}z_iz_j + o(|z|^2 +
|s-\bs|).
 \end{multline}

 In view of Definition  \ref{jet} we have

\begin{eqnarray}\label{contra1}
  4 \delta^{-1}(\bs-\bt)^3+2\lambda \bs^2\leq &&\sum_{ij=1}^m \bigg(\delta_{ij}
  -\frac{(k_i+\delta^{-1}f_i)(k_j+\delta^{-1}f_j)}{|k+\delta^{-1}f|^2}
  \bigg)\bigg|_{y=\by}
 (k_{ij}+\delta^{-1}f_{ij})\bigg|_{y=\by} \notag \\
 \leq && 2|(k_{ij}+\delta^{-1}f_{ij})| \notag\\
 \leq && C|(\by^{-1}\bx)_H| + O(\lambda^{1/r!}).
\end{eqnarray}
Substituting $y= \by,$ $s= \bs,$ in \eqref{previous} yields
\begin{multline}
 v_\mu(x, t)  \geq v_\mu(\bx, \bt) -
\alpha (t-\bt)
\\ - \delta^{-1}\big(|(x^{-1}
\by)_H|^4 - |(\bx^{-1} \by)_H|^4+ |\bs-t|^4- |\bs-\bt|^4\big)\\
+ \lambda (|\by|^{2r!}-|y|^{2r!}+|\bt|^2-|t|^2).\\
 \end{multline}
Setting $z= \bx^{-1}x$ and arguing as above we obtain
\begin{multline}
 v_\mu(\exp(zX)(\bx), t) \\ \geq v_\mu(\bx, \bt) -
\alpha (t-\bt) - \delta^{-1}\big(\sum_{i=1}^mf_iz_i +
\sum_{ij=1}^m f_{ij}z_iz_j \big) + 4
\delta^{-1}(\bs-\bt)^3(t-\bt)\\ - 2\lambda \bt (t-\bt)-
\big(\sum_{d(i)=1}^2k_iz_i + \sum_{ij=1}^m k_{ij}z_iz_j\big) +
o(|z|^2 +
|t-\bt|).\\
 \end{multline}
By Definition \ref{jet} it follows that
\begin{multline}\label{contra2}
-(\alpha -4 \delta^{-1}(\bs-\bt)^3+2\lambda \bt^2)
\\  \geq - \sum_{ij=1}^m \bigg(\delta_{ij}
-\frac{(k_i+\delta^{-1}f_i)(k_j+\delta^{-1}f_j)}{|k+\delta^{-1}f|^2}
\bigg)\bigg|_{x=\bx}
 (k_{ij}+\delta^{-1}f_{ij})\bigg|_{x=\bx}
 \end{multline}
Consequently
 \begin{multline}\alpha \leq 4 \delta^{-1}(\bs-\bt)^3
-2 \lambda \bt + c(|k_{ij}|+\delta^{-1}|f_{ij}|)\\  \leq
 4 \delta^{-1}(\bs-\bt)^3 + C\delta^{-1}|(\bx^{-1}\by)_H|^2 + O(\lambda^{1/r!})
\end{multline} In conclusion, using \eqref{contra1}
we have $$\alpha \leq 2 C\delta^{-1}|(\bx^{-1}\by)_H|^2,$$ for
 $\lambda$ sufficiently small.

\medskip {\bf 5.}
In view of Lemma \ref{supinf} the function $$\Phi(x,y,t,s)+
C|x,y,t,s|^2_E$$ is convex in the Euclidean sense in a
neighborhood of $(\bx,\by,\bt,\bs)$, which is a maximum point for
of $\Phi(x,y,t,s)$. Using Jensen's Lemma \ref{jensen} we see that there exists
points $(x^k, y^k, t^k, s^k)$ such that
\begin{equation}\label{ES3.24}
(x^k, y^k, t^k, s^k) \rightarrow (\bx, \by, \bt, \bs)
\end{equation}
\begin{equation}\label{ES3.25}
\Phi, u^\mu, v_\mu \text{ are twice differentiable in the
Euclidean}\footnote{We denote derivatives in the Euclidean metric
with the letter $D_E$}\text{  sense at } (x^k, y^k, t^k, s^k)
\notag
\end{equation}
\begin{equation}\label{ES3.26}
D_{E,x,y,t,s}\Phi (x^k, y^k, t^k, s^k)\rightarrow 0,
\end{equation}
\begin{equation}\label{ES3.26b}
D^2_{E,x,y,t,s}\Phi (x^k, y^k, t^k, s^k)\leq o(1) I_{2n+2}.
\end{equation}
From \eqref{ES3.26} we immediately deduce that
\begin{equation}
\nabla u^\mu(   y^k,  s^k)\rightarrow p + \lambda \nabla
(|y|^{2r!})_{|y=\by}
\end{equation}
\begin{equation}
\nabla v_\mu(x^k,t^k)\rightarrow p - \lambda \nabla
(|x|^{2r!})_{|x=\bx}
\end{equation}
where $$ p= 4 \delta^{-1}|(\by^{-1}\bx)_H |^2(\by^{-1}\bx)_H \neq
0.$$ Moreover
\begin{equation}
\p_s u^\mu(y^k,  s^k)\rightarrow q +2\lambda\bs,\quad \p_t
v_\mu(x^k, t^k)\rightarrow q - \alpha-2\lambda \bt
\end{equation}
with $$q \equiv 4 \delta^{-1}|(\bs-\bt) |^2(\bs-\bt).$$ On the
other hand  $D_{E,x,y}^2\Phi= A_1 + A_2$ where $$A_1 =D_{E,x,y}^2
(u^{\mu} - v_\mu - \delta^{-1}|(y^{-1}x)_H|^4) $$ and
$$A_2 = D_{E,x,y}^2 (\lambda |x|^{2r!} + \lambda |y|^{2r!})$$
In view of \eqref{ES3.26b} we have that at the point $(x^k,y^k,t^k,s^k)$,
$$A_1 \leq o(1) I_{2n+2} - A_2.$$
If we denote with $A$ the Hessian in the $x$ variable of
$|(y^{-1}x)_H|$, then for every $w \in \R^n$
\begin{multline}\label{matrix1}
(w, w) A_1
\begin{pmatrix}
  w  \\
  w
\end{pmatrix} \\ = (w, w)\begin{pmatrix}
 D_{E,y}^2u^\mu(y,s) -A & A \\
  A &  - D_{E,x}^2v_\mu(x,t)
\end{pmatrix}
\begin{pmatrix}
  w  \\
  w
\end{pmatrix} \\ = <(D_{E,y}^2u^\mu(y,s) - D_{E,x}^2v_\mu(x,t))w ,w>.
\end{multline}
Using \eqref{ES3.26b} and Lemma \ref{l1} it follows
that\footnote{We denote with $A^*$ the matrix $(A+A^T)/2$}
$$R^k-\bar R^k \leq o(1) I_{m} - \nabla^2 (\lambda |x^k|^{2r!} + \lambda |y^k|^{2r!})^*,$$
where
$$R^{k}=\nabla^2 u^\mu(y^k, s^k)^*, \quad  \bar R^k = \nabla^2 v_\mu(x^k,
t^k)^* .$$Using Lemma \ref{supinf} and passing to a subsequence if
necessary we see that there exist $m\times m $ matrices $R, \bar
R$ such that $R^k\to R, \bar R^k \to \bar R$ and $$R- \bar R \leq
(\lambda |\bx|^{2r!} + \lambda |\by|^{2r!})= O(\lambda ^{1/r!})I_m .$$
Using the fact that $u^\mu$ is a subsolution and $v_\mu$ is a
supersolution, and passing to the limit
\begin{multline} \label{matrix2} q +
O(\lambda ^{1/r!})\leq \big(\delta_{ij} - \frac{(p_i + O(\lambda
^{1/r!}))(p_j + O(\lambda ^{1/r!}))}{|p+O(\lambda ^{1/r!})|^2} \big)R_{ij}\\
\text{and} \quad q -\alpha + O(\lambda ^{1/r!})\geq
\big(\delta_{ij} - \frac{(p_i + O(\lambda ^{1/r!}))(p_j+ O(\lambda
^{1/r!}))}{|p+O(\lambda ^{1/r!})|^2} \big)\bar
R_{ij}.\end{multline} Subtracting, for $\lambda$ sufficiently
small, we obtain a contradiction, and complete the proof.
\end{proof}

Next we turn our attention to the special case of evolving graphs
$$u(x,e,t)=e-U(x,t)$$
in product groups of the form $\tilde G\times \R$. As we have seen,
$u$ solves \eqref{MAIN} if and only if $U$ solves \eqref{MAINbis}.

\begin{thrm}\label{ESth32bis}
Assume that $U$ is a bounded weak subsolution and $V$ is a bounded
weak supersolution of \eqref{MAINbis}. Suppose further

(i) For all $x\in \tilde G$ $U(x,0)\leq V(x,0).$

(ii) Either $U$ or $V$ is uniformly continuous when restricted to

$\tilde G \times \{t=0\}$. Then $U(x,t)\leq V(x,t)$ for all $x\in
\tilde G$ and $t\ge 0$. In particular, bounded weak solutions of
\eqref{MAINbis} are unique.
\end{thrm}
\begin{rmrk}
For bounded domains and in the special case of the Heisenberg group
this theorem follows from the results of Bieske \cite{bieske2}.
See also the comparison principle for the Gauss curvature flow
established in \cite{Haller}.
\end{rmrk}

\begin{proof}
We follow closely the steps in the proof of Theorem \ref{ESth32}
and outline only the main differences.  Arguing by contradiction
one easily sees that the function
$$\Phi(x,y,t,s)=U^\mu(y,s)-V_\mu(x,t)-\alpha t
-\frac{1}{\delta}\bigg( |yx^{-1}|_{E}^4+|s-t|^4\bigg)
-\lambda \bigg( |x|^{2r!}+|y|^{2r!}+|t|^2+|s|^2\bigg),$$
has a strictly positive maximum at the point $(\bx,\by,\bt,\bs)$ with
$$\bt,\bs \le c\sqrt{\mu},$$
$$\lambda  \bigg( |x|^{2r!}+|y|^{2r!}+|t|^2+|s|^2\bigg)\le C,$$
and
$$|\by\bx^{-1}|_E^4, |\bs-\bt|^4 \le C\delta.$$

Next we invoke  Jensen's Lemma \ref{jensen} and obtain a sequence of
points $(x^k, y^k, t^k, s^k)$ such that
\eqref{ES3.24},\eqref{ES3.25},\eqref{ES3.26}, and \eqref{ES3.26b} hold. In such points we obviously have
$$\nabla_0^x \Phi=-\nabla_0 V_\mu-\frac{1}{\delta}\nabla_0^x (|yx^{-1}|_E^4) -\lambda \nabla_0 |x|^{2r!},$$
and
$$\nabla_0^y \Phi=\nabla_0 U^\mu-\frac{1}{\delta}\nabla_0^y (|yx^{-1}|_E^4) -\lambda \nabla_0 |y|^{2r!}.$$

Next we observe that for any  differentiable function $f:\tilde G\to
\R$ and for any left invariant vector field $Z$ one has
\begin{equation}\label{diff-quot}
Z^x f(yx^{-1})= - Z^yf(yx^{-1}= \frac{d}{ds}f (ye^{-sZ}x^{-1})|_{s=0}.
\end{equation}
Similarly, if $f$ is twice differentiable and $W$ is another
left invariant vector field then
$$Z^xW^xf(yx^{-1})=Z^yW^yf(yx^{-1})=-Z^xW^yf(yx^{-1})=-Z^xW^yf(yx^{-1})$$

We immediately deduce that
\begin{equation}
\nabla U^\mu(   y^k,  s^k)\rightarrow p + \lambda \nabla_0
(|y|^{2r!})_{|y=\by}
\end{equation}
\begin{equation}
\nabla V_\mu(x^k,t^k)\rightarrow p - \lambda \nabla_0
(|x|^{2r!})_{|x=\bx}
\end{equation}
where $$ p= \nabla_0^y (|yx^{-1}|_E^4).$$ Moreover
\begin{equation}
\p_s u^\mu(y^k,  s^k)\rightarrow q +2\lambda\bs,\quad \p_t
v_\mu(x^k, t^k)\rightarrow q - \alpha-2\lambda \bt
\end{equation}
with $$q \equiv 4 \delta^{-1}|(\bs-\bt) |^2(\bs-\bt).$$

Note that, unlike for the PDE \eqref{MAIN}, here we do not have
to prove that $p\neq 0$, as \eqref{MAINbis} does not degenerate
with the vanishing of the gradient of its solution.

Using the computations above it is fairly straightforward to reproduce the
argument in \eqref{matrix1}-\eqref{matrix2} and thus conclude the proof
of the theorem.

\end{proof}

\section{Construction of barriers}\label{secbarrier}

In this section we construct explicit bounded weak solutions of
\eqref{PDE}, which we later use as barrier functions in the proof of
the existence theorem.

\subsection{Self-shrinking cylinder}\label{cilindro}

Let
\begin{equation}\label{u0}
u_0(x,t)= \frac{|x_H|^2}{2} + (m-1)t.
\end{equation} This function
depends only on the first layer variables and the mean curvature
operator, restricted to this layer reduce to the Euclidean mean
curvature operator in $\R^m$. The function $u_0$ satisfies
\eqref{MAIN} away from the characteristic set $\{0\}\times
V^2\oplus \cdots V^r$ (it is actually a weak solution in all of
$G$).
 The level sets $M_t =\{x:
u_0(x,t)=\frac{R_0}{2}\}$ of this function are products of a
sphere evolving by Euclidean mean curvature flow in $V_1$ with
initial data $\p B(0, R_0)$, with the higher layers $V_2 \oplus
\cdots \oplus V_r$. Note that the classical evolution is defined
up to time $\frac{R^2_0}{2(m-1)}.$ Moreover $M_t$ do not contain
any characteristic point and constitute a self-similar flow, i.e.
$M_t=\delta_{\lambda(t)}M_0,$ with $\lambda=
\sqrt{\frac{R_0^2}{2}-t(m-1)}$.

\subsection{Coordinate planes are equilibrium solutions}
Our goal here is to show that the coordinate planes $x_i=0$,
$d(i)=1,2$ are minimal surfaces, i.e. their mean curvature
vanishes identically outside of their characteristic set.

\begin{rmrk}
The result is false if $d(i)=3$ as one can easily see by examining
the plane $x_4=0$ in the Engel group \cite{CorwinGreenleaf}. This
group is best described in terms of its Lie algebra stratification
$\lie = V_1 \oplus V_2 \oplus V_3,$ where the dimension of $V_1$
is $2$ and the dimension of $V_2$ and $V_3$ is $1$. The algebra
has a system of generators $X_1, X_2\in V_1$ satisfying $[X_1,
X_2]=X_3 \in V_2,$  $[X_1, X_3]= X_4 \in V_3$ and all the other
commutators vanish. A possible representation of these vector
fields in coordinates $(x_1, \cdots, x_4)$ is $$X_1 = \p_{x_1} -
\frac{1}{2}x_2 \p_{x_3} - \Big(\frac{x_3}{2} + \frac{x_1
x_2}{12}\Big)\p_{x_4}, \quad X_2 = \p_{x_2} + \frac{1}{2}
x_1\p_{x_3} + \frac{1}{12} x_1^2\p_{x_4}$$
$$X_3 = \p_{x_3} + \frac{x_1}{2}\p_{x_4},\text{ and }\quad
X_4 = \p_{x_4}.$$ A direct computation yields
$$K_0= -\Bigg(\Big(\frac{x_3}{2} + \frac{x_1
x_2}{12}\Big)^2+\frac{x_1^4}{144} \Bigg)^{-3/2}
\frac{x_1^3x_3}{144},$$ away from the characteristic points.
\end{rmrk}

The starting point of our argument is the expression \eqref{bch}
for the vector fields $X_i$, $d(i)=1$ in terms of exponential
coordinates
$$X_i=\p_{x_i}+\sum_{d(j)=1, d(h)=2} c_{ij}^h x_j \p_{x_h} + \text{ higher order terms }.$$
 The Campbell-Hausdorff formula
 implies  the anti-symmetry relation $c_{ij}^h =
-c_{ji}^h$. It is immediate to observe that, if $d(k)=2$ one
has
\begin{equation}\label{gradplanes}
X_i (x_k)= \sum_{d(j)=1}c_{ij}^kx_j, \text{ and }X_i X_j
(x_k)=c_{ji}^k, \text{ for }d(i)=d(j)=1.
\end{equation}
Set $u(x)=x_k$, $d(k)=2$ then $|\nabla_0 u|^2\le
C(x_1^2+\cdots+x_m^2)$ and $(X_i X_j u)^*=0$. Consequently,
$$\sum_{i,j=1}^m \bigg(\delta_{ij} - \frac{X_i u_k X_j
u}{|\nabla_0 u|^2}\bigg) X_i X_j u=0,$$ if $|\nabla_0 u|\neq 0$.

Let us explicitly note that all the barriers $u_0$ (as in Section \ref{u0}),  $u_k=x_{k}^2$
(with $d(k)=2$) we have constructed so far, satisfy the following
properties
\begin{itemize}\label{hypo}
\item[(H1)] $u_k$ are solutions of the equation \eqref{MAIN} in $\{x\in G| |\nabla_0 u_k|\neq 0\}\times (0,\infty).$
\item[(H2)] $u_k$ are subcaloric (i.e. $\p_t u_k \leq \sum_{i=1}^m X_i^2 u_k$) in $G\times (0,\infty)$.
\item[(H3)] For every $C>0$ there exists $\tilde C>0$ such that
if $|x_H|,  |u_k|\leq C$ then\footnote{Here we recall
that $ \nabla_{1}$ denotes the full Riemannian gradient in the metric $g_1$.}
$|\nabla_{1} u_k| +
\sum_{i,j=1}^n|X_iX_j u_k|\leq \tilde C $, $d(k)\le 2$.
\end{itemize}

\subsection{Bounded barriers}\label{bba}
Define the cut-off function $\psi:[0,\infty)\to \R$,
$$\psi(s) =
  \begin{cases}
    (s-2)^3& \text{if } 0\leq s \leq 2, \\
    0 & \text{if }2\leq s.
  \end{cases}
$$
Note that
\begin{equation}\label{strutturapsi}
-8\leq \psi\leq 0, \quad \psi'\geq 0,  \quad  |\psi''|\leq
C_1\sqrt{\psi'} \leq C_2,
\end{equation}
 Set $v_i(x,t)= \psi(u_i(x,t)),$ where $u_i$ are $C^2_E$ functions satisfying (H1)--(H2)
above.
\begin{lemma}\label{boringlemma}
Assume there exists $C>0$ such that
\begin{equation}\label{lemmabound}
\psi''(u_k)|\nabla_{1} u_k|\leq C \text{ and
}\psi'(u_k)\sum_{i,j=1}^n|X_iX_j u_k|\leq C .
\end{equation}
There exists $C_0=C_0(C_1,C_2,C)$ such that if we set
$w_i^\delta(x,t) = v_i(x,t) - C_0\sqrt{\delta}t,$ then for all
$x\in G$, $t > 0$ and $\e>0$ sufficiently small with respect to
$\delta$,
 one has
$$\p_t w^\delta_k \leq \sum_{i,j=1}^n \bigg(\delta_{ij} - \frac{\Xie
w_k^\delta \Xje w_k^\delta }{|\nabla_\e
w_k^\delta|^2+\delta^2}\bigg)\Xie\Xje w_k^\delta .$$
\end{lemma}
\begin{rmrk}
Note that in view of (H3), estimates  \eqref{lemmabound} hold for
$k=0$ with no further assumption. If $|x_H|\leq C$ then
\eqref{lemmabound} hold also for $d(k)=2$
\end{rmrk}
\begin{proof}
It suffices to show that
$$\p_t v_k -\sum_{i,j=1}^n \bigg(\delta_{ij} - \frac{\Xie
v_k \Xje v_k }{|\nabla_\e v_k|^2+\delta^2}\bigg)\Xie\Xje v_k \le
C_0 \sqrt{\delta}.$$ The left-hand side can be rewritten as
\begin{multline}
\psi'(u_k) \p_t u_k \notag \\-\sum_{i,j=1}^n
\bigg(\delta_{ij}-\frac{\psi'(u_k)^2\Xie u_k \Xje
u_k}{\psi'(u_k)^2|\nabla_{\e}u_k|^2+\delta^2}\bigg)
\big(\psi'(u_k)\Xie \Xje u_k +\psi''(u_k)\Xie u_k \Xje u_k\big)
\\
=\psi'(u_k) \p_t u_k- \sum_{d(i)=d(j)=1} \bigg( \cdots \bigg) -
\sum_{d(i)+d(j)>2}\bigg( \cdots \bigg)\\= \psi'(u_k) \p_t u_k+
S_1+S_2 .
\end{multline}
Now we distinguish two cases: If $|\nabla_0 u_k|=0$ then we have
\begin{equation}\label{gradzero}
\psi'(u_k) \p_t u_k+ S_1 = \psi'(u_k) \p_t u_k- \sum_{i=1}^m X_i^2
u_k \le 0.
\end{equation}
In case $|\nabla_0 u_k|\neq 0$ we decompose $S_1$ as follows
\begin{eqnarray}
S_1= -&&  \sum_{i,j=1}^m \bigg(\delta_{ij}- \frac{X_i u_k X_j
u_k}{|\nabla_0 u_k|^2}
+\frac{X_i u_k X_j u_k}{|\nabla_0 u_k|^2} \notag \\
-&& \frac{X_i u_k X_j u_k [\psi'(u_k)]^2}{|\nabla_0
u_k|^2[\psi'(u_k)]^2+\delta^2 }
+\frac{X_i u_k X_j u_k [\psi'(u_k)]^2}{|\nabla_0 u_k|^2[\psi'(u_k)]^2+\delta^2 } \\
- &&\frac{X_i u_k X_j u_k [\psi'(u_k)]^2}{|\nabla_\e
u_k|^2[\psi'(u_k)]^2+\delta^2 } \bigg) \big(\psi'(u_k)X_i X_j
u_k +\psi''(u_k)X_i u_k X_j u_k\big)\notag
\\ && = S_{11}+S_{12}+S_{13}. \notag
\end{eqnarray}
where
\begin{multline}
S_{11} =  -\sum_{i,j=1}^m \bigg(\frac{X_i u_k X_j u_k}{|\nabla_0
u_k|^2} - \frac{X_i u_k X_j u_k [\psi'(u_k)]^2}{|\nabla_\e
u_k|^2[\psi'(u_k)]^2+\delta^2 } \bigg) \psi'(u_k)
X_i u_k X_j u_k\\
=-\bigg( \frac{\delta^2+\psi'(u_k)^2 \e^2\sum_{d(i)>1} (X_i u_k)^2
}{|\nabla_\e u_k|^2\psi'(u_k)^2+\delta^2}\bigg)
\sum_{i,j=1}^m\frac{X_i u_k X_j u_k}{|\nabla_0
u_k|^2}\psi'(u_k)X_i X_j u_k \le 0.
\end{multline}
Where the last indequality follows from $\psi'\geq 0$, hypotheses
(H1) and (H2) coupled with the expression
\begin{multline}
-\sum_{i,j=1}^m\frac{X_i u_k X_j u_k}{|\nabla_0 u_k|^2} X_i X_j u_k \\
=\sum_{i,j=1}^m \bigg(\delta_{ij}- \frac{X_i u_k X_j
u_k}{|\nabla_0 u_k|^2} \bigg)
 X_i X_j u_k - \sum_{i=1}^m X_i^2 u_k \le 0
\end{multline}
Next we estimate
\begin{eqnarray}
S_{12} &&=-\sum_{i,j=1}^m \bigg(\frac{X_i u_k X_j u_k}{|\nabla_0
u_k|^2} -\frac{X_i u_k X_j u_k [\psi'(u_k)]^2}{|\nabla_0
u_k|^2[\psi'(u_k)]^2+\delta^2 }\bigg) \psi''(u_k)X_i u_k X_j u_k
\notag\\ &&= -\psi''(u_k) \frac{\delta^2 |\nabla_0
u_k|^2}{|\nabla_0 u_k|^2[\psi'(u_k)]^2+\delta^2 }.
\end{eqnarray}
In view of \eqref{strutturapsi}, if $\psi'(u_k)\ge \delta$ one has
$$S_{12} \le \frac{\psi''(u_k) \delta^2}{\psi'(u_k)} \le C_0 \frac{\delta^2}{|\psi'(u_k)|^{3/2}} \le C_0 \sqrt{\delta}.$$
In case $\psi'(u_k) < \delta$ then from (H3) and
\eqref{lemmabound} we obtain
$$S_{12} \le |\psi''(u_k)| |\nabla_0 u_k|^2 \le C_0 \sqrt{\delta}.$$
Here we used the fact that $\psi''(u_k)=0$ if $|u_k|\geq 2$. If we
choose $\e^2\le \delta^{9/2}$ then
\begin{eqnarray}
S_{13} &&=-\sum_{i,j=1}^m \bigg(\frac{[\psi'(u_k)]^2 X_i u_k X_j
u_k}{|\nabla_0 u_k|^2[\psi'(u_k)]^2+\delta^2  } -\frac{X_i u_k X_j
u_k [\psi'(u_k)]^2}{|\nabla_\e u_k|^2[\psi'(u_k)]^2+\delta^2
}\bigg) \psi''(u_k)X_i u_k X_j u_k  \notag\\ &&=
-\psi''(u_k)\e^2\frac{[\psi'(u_k)]^4  |\nabla_0 u_k|^4
\sum_{d(i)>1}(X_i u_k)^2} {(|\nabla_0
u_k|^2[\psi'(u_k)]^2+\delta^2 )(|\nabla_\e
u_k|^2[\psi'(u_k)]^2+\delta^2 )}
\notag\\
&&\le C_2^5 C^6 \frac{\e^2}{\delta^4} \leq C_0 \sqrt{\delta}.
\notag
\end{eqnarray}
To conclude the proof we now estimate the higher layer derivatives
in $S_2$. Observing that $\e\le \delta^{9/4} \le \sqrt{\delta},$
one has
\begin{multline}
S_{2} =-\sum_{d(i)+d(j)>2} \psi''(u_k)
 \bigg(\delta_{ij}-\frac{\psi'(u_k)^2\Xie u_k \Xje
 u_k}{\psi'(u_k)^2|\nabla_{\e}u_k|^2+\delta^2}\bigg) \\
\big(\psi'(u_k)\Xie \Xje u_k +\psi''(u_k)\Xie u_k \Xje u_k\big)  =
O(\e) =O(\sqrt{\delta}). \notag
\end{multline}
\end{proof}


\section{Existence of weak solutions}
In this section we prove the existence of weak solutions to the
initial value problem for \eqref{MAIN}. Such solution will arise
as limit of solutions of regularized parabolic equations.

For $\delta,\sigma>0$, for all $\xi\in G$ and $1\le i,j\le n$ we
define the coefficients of the approximating equations
$$A_{ij}^{\e,\delta}(\xi)= \bigg(\delta_{ij}
-\frac{\xi_i \xi_j }{|\xi|^2+\delta}  \bigg),
$$
and
$$A_{ij}^{\e,\delta,\sigma} (\xi)= A_{ij}^{\e,\delta}(\xi) +\sigma \delta_{ij}.$$

\begin{prop}\label{delta}
For any $f\in C^{\infty}(G)$ there exists a unique solution
$\ued\in C^{\infty}(G)\times(0,\infty))$ of the initial value
problem
\begin{multline}\label{d-eq}
\frac{\p}{\p t} \ued= \sum_{i,j=1}^{n}
A_{ij}^{\e,\delta}(\nabla_{\e}\ued)X_i^{\e}X_j^{\e}\ued  \text{ in
}x\in G, t>0,
\\ \text{ and }\ued(x,0)=f(x) \text{ for all }x\in G.
\end{multline}
Moreover, for all $t>0$ one has
\begin{eqnarray}\label{d-bounds}
||\ued(\cdot, t)||_{L^{\infty}(G)} && \le
||f||_{L^{\infty}(G)} \notag \\
||\tnabla_{\e} \ued(\cdot, t)||_{L^{\infty}(G)} &&\le
||\tnabla_{\e} f||_{L^{\infty}(G)}. \notag
\end{eqnarray}
\end{prop}

\begin{cor}\label{graddecrease}
Let $u,f$ be as in the statement of Theorem \ref{delta}.
For any compact set $K\subset G$ there exists $C=C(K,G)>0$ such that if $0\le \e<1$,
\begin{equation}\label{d-li-bounds}
||\nabla_{\e} \ued(\cdot, t)||_{L^{\infty}(K)} \le C ||\nabla_E
f||_{L^{\infty}(G)} .
\end{equation}
\end{cor}

Ilmanen \cite[page 685]{MR1189906} shows that there exists a
unique smooth solution $\ued$ to \eqref{d-eq} satisfying the
bounds
\begin{eqnarray}\label{Ilmanen-bounds}
||\ued(\cdot, t)||_{L^{\infty}(G)}  && \le
||f||_{L^{\infty}(G)}\notag \\
||\p_t \ued(\cdot, t)||_{L^{\infty}(G)} &&
 \le C||X_iX_j f||_{L^{\infty}(G)} \notag\\
||\nabla_{\e} \ued(\cdot, t)||_{L^{\infty}(G)} && \le e^{-\lambda
t} ||\nabla_{\e} f||_{L^{\infty}(G)}, \notag
\end{eqnarray}
where $\lambda$ denotes the lowest eigenvalue for the Ricci tensor
of the Riemannian metric $g_{\e}$. A direct computation (see
\cite{CPT} for details) shows that $\lambda=-\frac{1}{\e^2}$. As a
consequence the estimates \eqref{d-bounds} which are uniform in
$\e$ do not follow immediately from \eqref{Ilmanen-bounds}.

\begin{proof}
We follow the outline of the analogue Euclidean result proved in
\cite[Theorem 4.1]{es:mc1}. For $\sigma>0$ we consider smooth
solutions\footnote{Existence and uniqueness are guaranteed by
classical parabolic theory \cite{MR0241822}} $\ueds$ of the
equation

\begin{equation}\label{ds-eqtn}
\frac{\p}{\p t} \ueds= \sum_{i,j=1}^{n}
A_{ij}^{\e,\delta,\sigma}(\nabla_{\e}\ueds)X_i^{\e} X_j^{\e}\ueds,
\end{equation}
with initial data $\ueds(x,0)=f(x)$, for all $x\in G$. In view of
the maximum principle we obtain
\begin{equation}\label{ds-1}
||\ueds(\cdot, t)||_{L^{\infty}(G)}  \le ||f||_{L^{\infty}(G)}.
\end{equation}
Since $\tilde X_1,....,\tilde{X}_{n}$ commute with the
left-invariant vector fields $X_1,...,X_{n}$ then we can
differentiate \eqref{ds-eqtn} along these directions and obtain
the new equation

\begin{multline}\label{ds-eqtn-deriv}
\frac{\p}{\p t} w= \sum_{i,j=1}^{n}
\bigg[A_{ij}^{\e,\delta,\sigma}(\nabla_{\e}\ueds )X_i^{\e}
X_j^{\e}w   \\ + \bigg(\p_{\xi_k} A_{ij}^{\e,\delta,\sigma}\bigg)
(\nabla_{\e}\ueds )X_i^{\e} X_j^{\e}\ueds X_k w\bigg],
\end{multline}
where we have let $w=\tilde{X}_i \ueds$, for all $i=1,...,n$. The
``{\it elliptic}" maximum principle applied to
\eqref{ds-eqtn-deriv} yields
\begin{equation}\label{ds-23}
||\tnabla_{\e}\ueds(\cdot, t)||_{L^{\infty}(G)}  \le
||\tnabla_{\e} f||_{L^{\infty}(G)}.
\end{equation}
Since the right invariant vector fields $\{ \tilde X_1,\cdots,
\tilde X_n \}$ form a basis of the tangent bundle of $G$, estimate
\eqref{ds-23} implies that

\begin{equation}\label{ds-2}
||\nabla_{\e}\ueds(\cdot, t)||_{L^{\infty}(G)}  \le C
||\tnabla_{\e} f||_{L^{\infty}(G)},
\end{equation}

for some positive constant $C$ depending only on $G$.

As remarked in \cite{es:mc1} the equation \eqref{ds-eqtn-deriv}
satisfies coercivity conditions
$$\bigg(1-\frac{M^2}{M^2+\delta} \bigg)|\xi|^2 \le
\sum_{i,j=1}^{n}  A_{ij}^{\e,\delta,\sigma}(\xi) \xi_i \xi_j \le 3
|\xi|^2,$$
 uniformly in $\sigma>0$ and provided $|\xi|\le M$. Classical parabolic
regularity theory (see \cite{MR0241822}) yields estimates on all
derivatives of $\ueds$ which are uniform in $0<\sigma<1$. To
conclude the proof we use \eqref{ds-1} and \eqref{ds-2},
Ascoli-Arzela' convergence theorem
 and Ilmanen's uniqueness
result to show that $\ueds\to \ued$ uniformly in the $C^1$ norm on
compact sets as $\sigma\to 0$.
\end{proof}

 Next, we need to extend to our setting Evans and Spruck's argument in
the proof of  \cite[Theorem 4.2]{es:mc1}. The difficulty here is
that we have two parameters rather than one. To our advantage we
have the fact that estimates \eqref{d-bounds} are {\it stable}
with respect to both $\delta\to 0$ and $\e\to 0$.

\begin{thrm}\label{visc-ex}
For any bounded $f\in C(G)$ there exists a viscosity solution
$u\in C(G\times(0,\infty))$ of
\begin{equation}\label{mce}
\p_t u=\sum_{i,j=1}^{m} A_{ij}^{0,0}(\nabla_0 u)X_iX_j u \ \text{
in }G\times (0,\infty)\ \text{ and }u(x,0)=f(x).
\end{equation}
\end{thrm}

Let $A^{\e}=(a_{ij}^{\e})$ be the matrix of coefficients of
$X_1^\e,...,X_{n}^\e$ in exponential coordinates,
  i.e.
$X_i^{\e}=\sum_{k=1}^{n}a_{ik}^{\e} \p_{x_k}$.

\begin{lemma}\label{l1}
Let $w$ be $C^2$ and such that at $(x_0,t_0)$ one has $D^2_Ew\le
0$ and $\nabla_E w= 0$, then $$(X_i^{\e}X_j^{\e} w)^* =
\frac{X_i^{\e}T_j^{\e} w +X_j^{\e}X_i^{\e} w}{2} \le 0.$$
\end{lemma}

\begin{proof}

A direct computation shows that
$$(X_i^{\e}X_j^{\e} w)^*= \sum_{lk=1}^{n} a_{il}^{\e}a_{jk}^{\e}\p_{x_k} \p_{x_l}w.$$
Hence, for all $\eta\in \R^{n}$ one has
$$\sum_{i,j=1}^{n} (X_i^{\e}X_j^{\e} w)^* \eta_i \eta_j=
 \sum_{lk=1}^{n}  (D^2_E w)_{lk} ([A^{\e}]^T\eta)_l ([A^{\e}]^T\eta)_k \le 0.$$
\end{proof}

\begin{proof}[Proof of Theorem \ref{visc-ex}]
Without loss of generality we can assume that $\nabla_E f$ is
bounded. The general case follows as in \cite[p. 659]{es:mc1}. Let
$\e_k,\delta_k\to 0$ be two sequences of positive numbers such
that $\e_k/\delta_k\to0$. In view of \eqref{d-bounds} it is
possible to find a sequence (corresponding to subsequences of
$\e_k$ and $\delta_k$) $u^k=u^{\e_k,\delta_k}$ of smooth solutions
to \eqref{d-eq}, with initial data $f$ and such that there exists
a locally Lipschitz (with respect to the Euclidean distance)
function $u$ such that $u^k\to u$ uniformly on compact sets.
Following the argument in \cite[Theorem 4.2]{es:mc1} we first show
that $u$ is a viscosity solution of \eqref{mce} and then prove
that it is constant in a set of the form $\{|x|+t\ge R\}$, with
$R$ depending on $K$.

Consider $\phi\in C^{\infty}(G\times(0,\infty))$ such that
$u-\phi$ has a local {\it strict} maximum point at $(x_0,t_0)$.
The uniform convergence $u^k\to u$ implies that there exists a
sequence of points $(x_k,t_k)\to (x_o,t_0)$ such that $u^k-\phi$
has a  local maximum at $(x_k,t_k)$. In particular
$$\nabla_E u^k=\nabla_E \phi, \ \p_t u^k=\p_t \phi, \ \text{ and }D^2_E(u^k-\phi)\le 0 \ \text{ at }(x_k,t_k).$$
In view of Lemma \ref{l1} we have that at the point $(x_k,t_k)$,
\begin{multline}\label{evs1}
\p_t \phi -A_{ij}^{\e_k,\delta_k}(\nabla_{\e_k} \phi)X_i^{\e_k}X_j^{\e_k} \phi \\
\le \p_t u^k  -A_{ij}^{\e_k,\delta_k}(\nabla_{\e_k} u^k)
X_i^{\e_k}X_j^{\e_k} (u^k+\phi-u^k) \le 0.
\end{multline}
If $\nabla_0 \phi(x_0,t_0)\neq 0$ then we simply take the limit as
$k\to \infty$ in \eqref{evs1} and conclude that $u$ satisfies
condition \eqref{compabove} in the definition of viscosity
subsolution. If $\nabla_0 \phi(x_0,t_0)=0$ then we set
$$\eta^k=\frac{\nabla_{\e_k} \phi(x_k,t_k)}{\sqrt{|\nabla_{\e_k}
\phi(x_k,t_k)|^2+\delta_k^2}}.$$ There exists $\eta\in \R^{n}$
such that $\eta^k\to \eta$. Notice that for $j=m+1,...,n$ one has
\begin{multline}|(\eta^k)_{j}|  = \frac{\e_k| X_j\phi(x_k,t_k)|}{\sqrt{|\nabla_{\e_k}
\phi(x_k,t_k)|^2+\delta_k^2}}
\\ \le\frac{ (\e_k/\delta_k) |X_j\phi(x_k,t_k)|}{\sqrt{(\e_k/\delta_k)^{2} \sum_{d(i)>1} (X_i  \phi(x_k,t_k))^2 + 1}}.
\end{multline}
Since this expression vanishes as $k\to \infty$ we have
 $\eta_{j}=0$ for $j=m+1,...,n$.
The PDE  \eqref{evs1} now reads as
$$\p_t \phi(x_k,t_k) - \sum_{i,j=1}^{n}(\delta_{ij}-\eta_i^k \eta_j^j) X_i^{\e_k} X_j^{\e_k} \phi(x_k,t_k) \le 0,$$
then as $k\to \infty$ we obtain
\begin{equation}\label{evs2}
\p_t \phi(x_0,t_0) \le  \sum_{i,j=1}^{m}
(\delta_{ij}-\eta_i\eta_k) X_iX_j \phi(x_0,t_0),
\end{equation}
concluding the proof in the case in which $u-\phi$ has a local
{\it strict} maximum point at $(x_0,t_0)$. If the maximum point is
not strict we argue as in \cite{es:mc1} and repeat the argument
above with $\phi$ replaced by
$$\tilde\phi(x,t)=\phi(x,t)+|x_0^{-1}x|^{2r!}+|t-t_0|^4,$$
Using Lemma \ref{l1} and repeating the previous argument one can
prove the  analogue of  \eqref{evs1} or \eqref{evs2} and from
there reaching the conclusion.
\end{proof}

\begin{rmrk}
If $|\tilde \nabla f| \leq C$ and $u$ is a viscosity solution of
the initial value problem in Theorem \ref{visc-ex}, then
\begin{equation}\label{lip}
\tLip(u) (\cdot, t) = \sup_{x\in G, h\in \R \text{ with }  h\not= 0}
\frac{u(\exp(h \tilde X)(x), t) - u(x,t) }{|h|}\leq C.
\end{equation}
\end{rmrk}

\begin{thrm}\label{compactdata}
Let $G$ be a Carnot group of step two. If we assume that the
function $f\in C(G)$ is constant in a neighborhood $G\setminus K$
 of infinity,
 then any weak solution $u$
of the initial value problem \eqref{mce} constructed as in
Theorem \ref{visc-ex} is constant in a set of
the form $\{|x|+t\ge R\}$, with $R$ depending on $K$.
\end{thrm}

\begin{proof}
Without loss of generality we can assume that the initial data $f$
satisfies $|f|\le 1$ in $G$ and $f(x)=0$ if $|x|>1$.
 Denote by
 $u_i$, $d(i)\le 2$,  the barrier functions constructed
in  Section \ref{secbarrier} and  $\psi$ the cut-off function
defined  in Section \ref{bba}. For all $x\in G$ and $t>0$ set
$v_i(x,t)= \psi(u_i(x,t)),$ and $w_i^\delta(x,t) = v_i(x,t) -
C_0\sqrt{\delta}t.$ In view of Lemma \ref{boringlemma} we have
that for all $x\in G$, $t > 0$ and $\e>0$ sufficiently small with
respect to $\delta$,
 one has
\begin{equation}\label{sotto}
\p_t w^\delta_k \leq \sum_{i,j=1}^n \bigg(\delta_{ij} - \frac{\Xie
w_k^\delta \Xje w_k^\delta }{|\nabla_\e
w_k^\delta|^2+\delta^2}\bigg)\Xie\Xje w_k^\delta ,
\end{equation}
in the set where \eqref{lemmabound} holds. Note that
$w_0^{\delta}(x,0)= \psi(|x_H|^2/2)=0$ for $|x_H|\ge 2$ and
$w_0^{\delta}(x,0) \le -1$ if $|x_H|\le 1.$ Also, observe that for
$d(k)=2$, we have $w_k^{\delta}(x,0)=0$ if $x_k^2\ge 2$ and
$w_k^{\delta}(x,0) \le -1$ if $x_k^2\le 1$. Let $\ued$ be as in
Proposition \ref{delta}, that is a solution of the approximating
equation with initial data $g$. Since $f(x)\ge w_0^{\delta}(x,0)$
for all $x\in G$, then in view of the classical comparison
principle for smooth solutions of quasilinear parabolic equations
(see \cite{MR0241822}) we have $\ued(x,t)\ge w_0^{\delta}(x,t)$
for all $x\in G$ and $t>0$. In view of the uniform convergence
proved above, if we let $\delta,\e\to 0$ we obtain $u(x,t)\ge
\psi(|x_H|^2/2+(m-1)t) \ge 0$ for $|x_H|^2/2+(m-1)t\ge 2$. An
analgous argument yields $u(x,t)=0$ in the set
$|x_H|^2/2+(m-1)t\ge 2$.

At this point we restrict our attention to the the region
$A=\{x\in G|\  |x_H|\le 2\}$, since we already know that $u(x,t)$
vanishes outside $A$ for every $t>0$. Note that \eqref{lemmabound}
holds for $w_k^{\delta}$, $d(k)=2$ in the set $A$. Applying Lemma
\ref{boringlemma} we obtain that $w_k^{\delta}$ satisfies
\eqref{sotto} in $A$. Since $f(x)\ge w_k^{\delta}(x,0)$ for all
$x\in G$, the classical maximum principle ensures that
$\ued(x,t)\ge w_k^{\delta}(x,t)$, for $x\in A, t>0$. Arguing as
above $u(x,t)\ge \psi(x_k^2)$ for $x\in A, t>0$. In particular
$u(x,t)\ge 0$ for $x \in A$ such that $x_k^2 \ge 2$. Similar
arguments, applied to $-u$, yields $u=0$ in the same set. In
conclusion $u(x,t)=0$ for all $(x.t)$ such that
$|x_H|^2/2+(m-1)t\ge 2$, $x_k^2 \ge 2$, $d(k)=2.$
\end{proof}

\section{Some geometric properties of the flow}
As we mentioned earlier, lacking a complete form of the comparison principle, we cannot prove that the generalized mean curvature flow defined
in Section 2, does not depend on the choice of the initial data
$f$, but only on its zero level set.
%
However  we can show two basic geometric properties for the flow, namely (i) separation property and (ii) show that the right invariant distance between level sets
is not increasing with time.

We say that a level set $M=\{u(x)=0\}$ is {\it cylindric } if $u(x_H,x_V)$
is constant in  the $x_V$ variables.

\begin{prop}
Let $M_0, \hat M_0$ be subset of $G$ and denote  by $M_t$ and
$\hat M_t$ the corresponding generalized flows. We have

(i) If $M_0\subset \hat M_0$  and $\hat M_t$, $t\ge 0$
is cylindric, then $M_t\subset \hat M_t$, for all
$t>0.$

(ii) For this part we consider the flows $M_t,$ $\hat M_t$
arising as level set of the solutions constructed in Theorem \ref{visc-ex}. If we denote by $\tilde d(\cdot, \cdot)$ the right invariant
CC distance, then $$\tilde d(M_0, \hat M_0)\leq \tilde d(M_t, \hat
M_t)$$ for all $t>0.$
\end{prop}
\begin{proof}
Part (i) is a direct consequence of the comparison principle. As
for (ii) assume that $\tilde d(M_0, \hat M_0)>0$. We recall
a result of Monti and Serra Cassano in \cite{msc:cc}, where
it is proved that  $|\tilde \nabla_0 \tilde d(\cdots, M_0)|=1$ outside $M_0$.
Thanks to this result it is immediate to construct a function $\tilde f$
such that $\tilde f=1$ on $M_0$, vanishes in $\{x\in G: \tilde d(x, M_0)> \tilde d(M_0, \hat M_0)\}$
and $|\tilde \nabla \tilde f|\leq \tilde d(M_0, \hat M_0)^{-1}$.
A simple modification of this construction yelds a
function $f$ such that $M_0$ is its zero level set,
$\hat M_0$ its 1- level set and $|\tilde \nabla_0 f|\leq \tilde d(M_0, \hat M_0)^{-1}$.
Let us denote by $u$ the unique weak solutions to
\eqref{MAIN} with initial data $f$, and denote
$M_t$  its zero level set, and
$\hat M_t$ its 1- level set. For each $t>0$ we choose points $x \in M_t$,
$\hat x \in \hat M_t$ such that $\tilde d(x, \hat x)= \tilde d(M_t, \hat M_t)$.
Using Corollary \ref{graddecrease} we have
$$1= |u(x,t) - u(\hat x,t)| \leq \tLip (u) \tilde d(x, \hat x) \leq
\tLip (g) \tilde d(x, \hat x) = \frac{\tilde d(M_t, \hat M_t)}{\tilde d(M_0, \hat M_0)},$$
concluding the proof.
\end{proof}

We recall that the self-similar cylinder, as defined in Section
\ref{cilindro}, vanishes in a finite time. As a corollary we deduce
that any compact set evolves within a shrinking cylinder and
vanishes in a finite time.

\bibliographystyle{acm}
\bibliography{GiLu}
\end{document}